\documentclass[preprint,12pt,authoryear]{elsarticle}

\usepackage{amssymb}

\usepackage{amsthm}
\usepackage{amsmath}
\usepackage{amsfonts}
\usepackage{mathrsfs}
\usepackage{cases}
\usepackage{caption}

\usepackage[colorlinks,linkcolor=red, anchorcolor=blue, citecolor=green]{hyperref}

\newdefinition{definition}{Definition}[section]
\newdefinition{example}[definition]{Example}
\newdefinition{remark}[definition]{Remark}

\newtheorem{lemma}[definition]{Lemma}
\newtheorem{theorem}[definition]{Theorem}

\newtheorem{proposition}[definition]{Proposition}

\numberwithin{equation}{section}

\journal{...} %

\begin{document}

\begin{frontmatter}

%% Title, authors and addresses

%% use the tnoteref command within \title for footnotes;
%% use the tnotetext command for theassociated footnote;
%% use the fnref command within \author or \address for footnotes;
%% use the fntext command for theassociated footnote;
%% use the corref command within \author for corresponding author footnotes;
%% use the cortext command for theassociated footnote;
%% use the ead command for the email address,
%% and the form \ead[url] for the home page:
%% \title{Title\tnoteref{label1}}
%% \tnotetext[label1]{}
%% \author{Name\corref{cor1}\fnref{label2}}
%% \ead{email address}
%% \ead[url]{home page}
%% \fntext[label2]{}
%% \cortext[cor1]{}
%% \address{Address\fnref{label3}}
%% \fntext[label3]{}

\title{Dividend Barrier Strategies in a Renewal Risk Model With Phase-Type Distributed Interclaim Times\tnoteref{t1}}
\tnotetext[t1]{This work was supported by National Natural Science Foundation of China Grants No. 11471171, No. 11911530091 and No. 11931018.}

\author[dhu]{Linlin Tian}\ead{linlin.tian@dhu.edu.cn}
\author[nku]{Zhaoyang Liu\corref{cor1}}\ead{liuzhy@nankai.edu.cn}
\cortext[cor1]{Corresponding author}
\address[dhu]{College of Science, Donghua University,
Shanghai, P.R. China 201620}
\address[nku]{School of Mathematical Sciences, Nankai University, Tianjin, P.R. China 300071}
\begin{abstract}
   In this paper, we consider the optimal  dividend problem of the renewal risk model with phase-type distributed interclaim times and exponentially distributed claim sizes.
  Assume that the phases of the interclaim times can be observed.  We study
  the optimal  dividend under the 2--order  and the $n$--order ($n\ge 3$)separately.    In the case of  2--order phase-type distributed interclaim times, we show that the optimal dividend policy is the optimal phase-wise barrier strategy. As a byproduct, we find that in this case, the phase with higher barrier is the phase with the higher intensity to the next claim.
   In the case of  n-order ($n\ge3$) phase-type distributed interclaim times,
    an iteration algorithm is presented to show that the optimal phase-wise barrier strategy is optimal among all dividend policies. We also find a similar conclusion like in the case of 2--order, the phase with the highest barrier is the phase with the highest intensity to the next claim.
\end{abstract}

\begin{keyword}
%% keywords here, in the form: keyword \sep keyword
Hamilton-Jacobi-Bellman equation \sep phase-type distribution \sep optimal dividend.
%% PACS codes here, in the form: \PACS code \sep code

%% MSC codes here, in the form: \MSC code \sep code
%% or \MSC[2008] code \sep code (2000 is the default)
\MSC[2010] 60J25 \sep 65K10 \sep 91B30 \sep 93E20
\end{keyword}

\end{frontmatter}

%% \linenumbers

%% main text
\section{Introduction}
\label{sec:Introduction}
The optimal dividend problem can be traced back to \cite{DeFinetti1957}.
\cite{AsmussenTaksar1997} studied this problem when the surplus process is modeled by a Brownian motion with drift.
The dividend optimization problem under the compound Poisson model is studied in \cite{AzcueMuler2005Optimal}, \cite{Belhaj2010} and \cite{GerberShiu2006}.
When the Poisson process is replaced by the renewal process, \cite{Albrecher2005} calculated the distribution of the discounted dividends for a barrier strategy when the interclaim times follow the generalized Erlang($n$)-distribution.
By numerical simulation, \cite{AlbrecherHartinger2006} showed that the horizontal dividend barrier strategy is not necessarily optimal.
Mishura and Schmidli \cite{MishuraSchmidli2012} showed that the phase-wise dividend barrier strategy is optimal when the interclaim times are Erlang($n$)-distributed and the claim sizes are exponentially distributed.
In this paper, we will extend their work and consider a renewal risk model where the interclaim times are phase-type distributed.

For a given $n$--order phase-type distributed interclaim time ($n\ge 2$ is a positive constant),
we assume  that the phases of the interclaim times can be observed. We focus on the optimal phase-wise barrier strategy. Now we introduce the definition of the  phase-wise barrier strategy first. For any given phase $i\in\{1,2,\ldots, n\}$, we set up barriers $\{b_i\}_{i=1}^n$ such that in any phase $i,$  all capital above $b_i$ is paid as dividend, if the capital is less than the barrier $b_i$, then the company will keep the premium until the occurrence of the next claim or reaching the given level $b_i$. This kind of strategy is called the phase-wise barrier strategy. The optimal phase-wise barriers $\{b_i^*\}_{i=1}^n$ is chosen among all phase-wise barriers $\{b_i\}_{i=1}^n$ such that for any initial  capital $x<b_i^*$, the  cumulative discounted discounted dividend at $x$ attains its maximum.

By the martingale approach and some direct calculations,  we get  some necessary conditions for the optimal phase-wise barriers. Then we deal with the case of $2$--order and $n(n\ge 3)$--order separately due to the complexity of the inverse of the $n\times n$ matrix   in the case of $n(n\ge 3)$--order. Now we explain those two parts separately.

In the case of $2$--order phase-type distribution, there are only two barriers $b_1^*, b_2^*$. If $b_1^*<b_2^*,$ The interval $[0,+\infty)$ is only separated into only three parts: $[0,b_1^*)\cup[b_1^*, b_2^*)\cup[b_2^*,+\infty).$ We can show that the optimal phase-wise barrier strategies' corresponding function is concave on $[b_1^*,+\infty).$ The way of comparing of the size of $b_1^*,b_2^*$ can also be derived, which is the phase with the higher barrier is the phase with the higher intensity to the next claim. Mathematically speaking, the necessary condition of $b_1^*<b_2^*$ is $\lambda_{23}<\lambda_{13}$, where $\lambda_{i3}$ ($i=1,2$) denotes the intensity of phase $i$ of entering the claim state. In the last, we can theoretically show that the optimal phase-wise barrier strategy is optimal among all admissible strategies via the  HJB equation.

For the case of $n(n\ge 3)$--order distributed interclaim times, we use the numerical method to explore the optimal strategy. We bring forward a  similar iteration algorithm used in \cite{Albrecher2017} and \cite{Liu2020Dividend} to show that the optimal strategy is a phase-wise barrier strategy, which is obviously  the optimal phase-wise barrier strategy. The proof of the convergence of the algorithm is given. Examples of optimal phase-wise  barriers  are presented to illustrate the applicability of numerical methods.
Interestingly, we also prove that the phase with the highest optimal barrier is the phase with the highest intensity to the next claim. This phenomenon coincides with the conclusion we get in the case of $2$--order.

The whole paper is organized as follows.
In Section \ref{sec:introducemodel}, we  establish the basic model and formulate the problem.
In Section \ref{sec:phasewisebarrier}, we study the optimal phase-wise barrier and get some necessary conditions for the optimal phase-wise barriers.
In Section \ref{sec:2ordercase}, we study the case of 2--order phase-type distributed interclaim times, the optimal phase-wise barrier strategy is the optimal strategy among all dividend strategies. As a byproduct, we find that  the phase with the higher barrier is the phase with the higher intensity to next claim.
In Section \ref{sec:nordercase}, we numerically study the optimal dividend problem under the  $n$($n\ge 3$)--order  phase-type distributed interclaim times, from which we find  that the optimal policy is the optimal phase-wise barrier strategy. Similar with the Section \ref{sec:2ordercase}, we also get a  conclusion that the phase with the phase with the highest optimal barrier is the phase with the highest intensity to the next claim. Some examples are given to show the application of the algorithm.

\section{Models and assumption}
\label{sec:introducemodel}
In this section, we present the surplus process of the insurer, which includes models for aggregate claims and dividend payments to policyholders, define the value function.
We work on a complete probability space $(\Omega,\mathscr{F}, \mathbb{P})$ on which all processes are well defined.
The information at time $t$ is given by $\mathscr{F}_t$, in which $\{\mathscr{F}_t: t \ge 0\}$ is the complete filtration generated by the claim and the dividend processes.

The surplus process of an insurance company with dividend payments is modeled as
\begin{equation*}
X_t^D=x+ct-\sum_{i=1}^{N_t}Y_i-D_t,
\end{equation*}
where $c>0$ is the premium rate, $x$ is the initial wealth, $D_t$ is the cumulative amount of dividends paid out up to time $t$.
$N_t$ is a simple counting process representing the number of  the incoming claims and $\{Y_i\}$ are i.i.d. with a common distribution $G(x):=1-e^{-\beta x} (\beta>0)$ and independent of $N_t$.
The interclaim times are independent and follow the phase-type distribution.

The phase-type distribution is the distribution of the life time of a terminating Markov process $\{J_t\}_{t\ge 0}$ with finitely many phases and time homogenous transition rates.
More precisely, let $\{\bar{J}_t\}_{t\ge 0}$ be a Markov process on the finite state space $E_\Delta=E\cup\{\Delta\}$,  where $E$ is the state space and $\Delta$ is the absorbing state.
The terminating Markov process $\{J_t\}$ with state space $E$ and intensity matrix $\mathbf{T}$ is defined as the restriction to $E$ of $\{\bar{J}_t\}$.
The Markov process $\{\bar{J}_t\}_{t\ge 0}$ jumps from one state to another.
All the states $i\in E$ are transient and once the Markov process $\{\bar{J}_t\}$ enters the absorbing state, then it will stay in this absorbing state forever.
See \cite{AsmussenAlbrecher2010} and \cite{Bladt2005} for more details about phase-type distribution.

In our model, the state space $E:=\{\mbox{state}\;1, \mbox{state}\;2,\cdots, \mbox{state}\;n\}$ and the $\Delta:=\mbox{state}\;n+1$, where $n$ is a positive constant.
We call this phase-type distribution a $n$-order phase-type distribution.
Thus, the intensity matrix of $\bar{J}_t$ has the form
\begin{equation*}
\mathbf{\Lambda}=
 \left(
           \begin{array}{cc}
             \mathbf{T} & \mathbf{t} \\
             \mathbf{0} & 0 \\
           \end{array}
         \right),
\end{equation*}
where $\mathbf{T}$ is $n\times n$ dimensional matrix, $\mathbf{t}$ is a $n$ dimensional column vector and $\mathbf{0}$ is the $n$ dimensional row vector of zeros.
In particular, $\mathbf{t}=-\mathbf{T}\mathbf{e}$, where $\mathbf{e}$ is the n-dimensional column vector with all components equal to $1$, which means, the intensity of leaving state $i$ equals to the sum of the intensities of leaving state $i$ and entering the new state $j$.

The Markov process $J_t$ jumps from one state (or we all state as ``phase") to another and stay in the state $i$ for an exponential time with parameter $\lambda_i$, $0\le i\le n,\lambda_i=-\lambda_{ii}>0$.
Once it enters the absorbing state $n+1$, the claim will occur.
After the claim, the Markov chain $J_t$ will restart at the state $i\in E$ with the initial probabilities $\pi_i$, $i=1,2\cdots, n$.
Here $\sum_{i=1}^{n}\pi_i=1.$
Then it continues to jump from one state to another until the next absorption (next claim). Now we assume that the Markov process $\{J_t\}$ can be observed and is independent of the claim sizes and $N$.

%In our model, the phase-type distribution is the distribution of a time homogenous Markov
%process $\{J_t\}_{t\ge 0}$ which takes value in $E:=\{1,2,\cdots, n\}$. The Markov process $\{J_t\}_{t\ge 0}$ jumps from one state to another and stays in state $i$ for an exponential time with parameter $\lambda_i$, $0\le i\le n,\lambda_i>0$. Denote $\mathbf{\Lambda}=\{\lambda_{ij}\}_{i,j\in E}$ its transition matrix. The absorption state space is denoted by $n+1$. The intensity of leaving state $i$ and entering the absorbing state $n+1$ is $\lambda_{i,n+1}$.
%It is worth noting that $\lambda_i=\sum_{j=1,j\neq i}^{n+1}\lambda_{i,j}$, which means the %intensity of leaving state $i$ equals to the sum of the intensity of  leaving state $i$ and %entering the new phase $j$.    Once the Markov process jumps out of the state space $E$, %which means the Markov process enters the absorbing state $n+1$, then the claim will occur.

$D_t$ is the cumulative amount of dividends paid out up to time $t$. We say that a dividend strategy $D_t$ is \emph{admissible} if
\begin{itemize}
\item  $D_t$ is predictable, nondecreasing, c${\grave{\mbox{a}}}$gl${\grave{\mbox{a}}}$d;
\item The process $D_t$ verifies $D_t\le x+ct-\sum_{i=1}^{N_t}Y_i$.
\end{itemize}
We denote by $\mathscr{U}_{ad}$ the set of all the admissible control strategies.
For any dividend strategy $D$, the expected discounted dividend payments is defined as
\begin{align*}
V_i^{D}(x)=\mathbb{E}\left[\int_{0-}^{\tau-}e^{-\delta t} \mathrm{d}D_t\Big|J_0=i,X_{0-}^D=x\right]\\:=\mathbb{E}_{ix}\left[\int_{0-}^{\tau-}e^{-\delta t} \mathrm{d}D_t\right], \quad i=1,2,\cdots, n,
\end{align*}
where $\delta>0$ is the discount factor, $\tau=\inf\{t:X_t^{D}<0\}$ is the time of ruin.
The optimal value function is defined as
\begin{equation}\label{phasevaluefunction}
V_i(x)=\sup_{D\in\mathscr{U}_{ad}}V_i^D(x) ,
\end{equation}
for all $x \ge0$.
In the following sections,
 we will show that optimal dividend policy is the optimal phase-wise barrier strategy. But we will treat the case of 2--order and the n--order separately. First, we study the optimal phase-wise barrier strategy.

\section{The optimal phase-wise barrier}
\label{sec:phasewisebarrier}
Assume the insurance company choose  a phase-wise barrier strategy, which means, they will choose a barrier $b_i\ge 0$ for a given phase $i$, $i=1,2,\cdots, n$.
If the Markov process $\{J_t\}$ is in state $J_t=i$, all the capital above $b_i$ is paid as a dividend.
If the wealth equals to $b_i$ when $J_t=i$, then all the incoming premium will be paid as dividends until the next jump of $J_t$ occurs.
Let
\begin{equation*}
f_i^D(x)=\mathbb{E}\left[\int_{0-}^{\tau-}e^{-\delta t} \mathrm{d}D_t\Big|J_0=i,X_{0-}^D=x\right], i=1,2,\cdots, n,
\end{equation*}
denote the phase-wise barrier strategy $D$'s cumulative discounted dividends, for simplicity we omit the symbol $D$   if there is no confusion.
For simplicity, we let $f_{n+1}(x)=\sum_{i=1}^n\pi_i\mathbb{E}[f_i(x-Y)]$, where $Y$ is a random variable with the distribution $G(\cdot)$.
Now we show that the functions $\{f_i(x)\}_{i=1}^n$ are continuously differentiable on $[0,b_i]$ and fulfill
\begin{equation}\label{barrier1}
cf_i'(x)+\sum_{j=1,j\neq i}^{n+1}\lambda_{ij}f_j(x)-(\lambda_i+\delta)f_i(x)=0, \quad x\le b_i,
\end{equation}
and
\begin{equation*}
f_i(x)=f_i(b_i)+x-b_i,\quad x\ge b_i.
\end{equation*}
Actually, for $0\le x<b_i$, we start by considering all possible events over a very small time interval $[0,h]$, we let $h$ be a small enough such that $x+ch<b_i$. Conditioning on the first jump of $\{J_t\}$, we arrive at
\begin{align}\label{toshowdiff1}
f_i(x)=e^{-\delta h}(1-\lambda_ih)f_i(x+ch)+\sum_{j=1,j\neq i}^{n+1}\int_0^he^{-\delta t}f_j(x+ct)\lambda_{ij}e^{-\lambda_it}dt+o(h).
\end{align}
Upon rearrangement of \eqref{toshowdiff1}, division by $h$ and letting $h\rightarrow 0$ yields
\begin{equation*}
cf_i'(x)+\sum_{j=1,j\neq i}^{n+1}\lambda_{ij}f_j(x)-(\lambda_i+\delta)f_i(x)=0, \quad x\le b_i.
\end{equation*}
Although the above derivative are understood to be right-derivatives, one can apply the same arguments with $x-ch$ in replace of $x$ to find that \eqref{barrier1} are also satisfied by the left-derivative.
Eventually, we get that $f_i(x)$  are continuously differentiable on $[0,b_i]$ and satisfies \eqref{barrier1}.

Conditioning on the first jump of $\{J_t\}$,
\begin{align}\label{barrier11}
f_i(b_i)=&\int_0^\infty\int_0^tce^{-\delta s}\mathrm{d}s \lambda_{i}e^{-\lambda_i t} \mathrm{d}t+\int_0^{\infty}\sum_{i=1,i\neq j}^{n+1}\lambda_{ij}f_j(b_i)e^{-(\lambda_i+\delta)t}\mathrm{d}t \nonumber\\
=&\frac{c+\sum_{i=1,i\neq j}^{n+1}\lambda_{ij}f_j(b_i)}{\lambda_i+\delta}.
\end{align}
Combining with (\ref{barrier1}), we conclude that $f_i'(b_i)=1.$

On the other hand,  we show that
a solution $f_i(x)$ to equations (\ref{barrier1}) on $[0,b_i]$ and $f_i(x)=f_i(b_i)+x-b_i$ for $x>b_i$ with $f_i'(b_i)=1$ is the expected cumulative discounted dividend value of the barrier strategy with barriers at $b_i$. Consider a phase-wise barrier  strategy with barriers $\{b_i\}_{i=1}^n$ and the corresponding surplus process is denoted by $X_t^b$. Denote the ruin time as $\tau.$ By the definition of phase-wise barrier strategy, we know that for any given phase $J_t=i, i=1,2,\cdots,n$,  the surplus $X_t^b\le b_i$. By the It\^{o} formula,
\begin{align*}
f_{J_t}(X_{\tau\wedge t}^b)e^{-\delta (\tau\wedge t)}=&f_i(x)+\int_0^{\tau\wedge t}e^{-\delta s}[-\delta f_{J_s}(X_s^b)+\mathscr{L}f_{J_s}(X_s^b)]\mathrm{d}s \\
&-\int_0^{\tau\wedge t}{f_{J_s}}'(X_s^b)e^{-\delta s}\mathrm{d}L_s,
\end{align*}
where
\begin{equation*}
\mathscr{L}f_i(x)=c{f_i}'(x)+\sum_{j=1,j\neq i }^{n+1}\lambda_{ij}{f_j}(x)-\lambda_{i}f_i(x)
\end{equation*}
is the infinitesimal generator of the process.
Since $X_t^b$ follows the phase-wise barrier strategy with barriers $\{b_i\}_{i=1}^n$, $f_i(x)$ solves \eqref{barrier1} on $[0,b_i)$ and $f_i'(x)=1$ on $[b_i,+\infty)$,  we have that
\begin{align*}
\mathbb{E}\left[f_{J_t}(X_{\tau\wedge t}^b)e^{-\delta (\tau\wedge t)}\right]=f_i(x)-\mathbb{E}\left[\int_0^{\tau\wedge t}e^{-\delta s}\mathrm{d}L_s\right].
\end{align*}
Using the dominated convergence theorem and letting $t\rightarrow\infty$, we get that $f_i(x)=\mathbb{E}\left[\int_0^{\tau}e^{-\delta s}\mathrm{d}L_s\right]$, completing the desired conclusion.

Now we try to analyse the optimal phase-wise  barriers  $\{b_i\}_{i=1}^n$ such that $\{f_i(x)\}_{i=1}^n$ become maximal for all $x<b_i$.
Denote the optimal phase-wise barrier as $\{b_i^*\}_{i=1}^n$ and the corresponding value function as $\{f_i^*\}_{i=1}^n$.
Given $1\le i\le n$ and  $x<b_i^*$,
\begin{align}\label{numerical1}
f_i^{*}(x)=&\int_0^{\frac{b_i^{*}-x}{c}}e^{-(\lambda_i+\delta)t} \Bigg(\sum_{\substack{j=1\\j\neq i}}^{n+1}\lambda_{ij}f_j^{*}(x+ct)\Bigg)\mathrm{d}t +e^{-(\lambda_i+\delta)(\frac{b_i^{*}-x}{c})}f_i^{*}(b_i^{*}).
\end{align}

For a given initial wealth $x<b_i^*$, consider a special strategy: paying the incoming premium as dividend until next phase jump.  After the jump, follow the ``optimal'' phase-wise barrier strategy. Then, by the definition of the optimal phase-wise barrier strategy, we see that
\begin{align}\nonumber
f_i^*(x)&>\int_0^\infty\Big(\int_0^tce^{-\delta s}\mathrm{d}s\Big)\lambda_i e^{-\lambda_i t}\mathrm{d}t+\int_0^\infty \Bigg(\sum_{\substack{j=1\\j\neq i}}^{n+1}\lambda_{ij}f_j^*(x)\Bigg)e^{-(\delta+\lambda_i)t}\mathrm{d}t\\\label{barrier2}
&=\frac{c+\sum_{j=1,j\neq i}^{n+1}\lambda_{ij}f_j^*(x)}{\lambda_i+\delta}.
\end{align}
Combining with (\ref{barrier1}), we see that ${f_i^*}'(x)>1$ on $(0,b_i^*).$
Now we show that ${f_i^*}''(b_i^*)=0$. Here we borrow some arguments from \cite{MishuraSchmidli2012}. Since $f_i^*(x)=f_i^*(b_i^*)+x-b_i^*$ on $[b_i^*, +\infty)$, ${f_i^*}''(b_i^*+)=0$. Thus, we only need to show that ${f_i^*}''(b_i^*-)$ exists and equals to $0.$ We consider a function $M_i(x)$ satisfies \eqref{barrier1} on $[0,+\infty)$ with the boundary condition $M_i(0)=f_i^*(0)$. Obviously we have  $M_i(x)=f_i^*(x)$ on $[0,b_i^*].$  We already know that $M_i(x)>1$ for $x<b_i^*$. Suppose that $M_i(x)<1$ for $x\in(b_i^*,b_i^*+\varepsilon]$ for some $\varepsilon>0$. Suppose the Markov process  $J_t$ start with $J_0=i.$  We consider the following strategy.
First use a barrier strategy with a barrier at $b_i^*+\varepsilon$ then we follow the strategy with barriers at $\{b_i^*\}_{i=1}^n$, except the first time we reach state $i$, where we use a barrier at $b_i^*+\varepsilon$.  Denote the expected cumulative discounted dividend  of this strategy by $\tilde{f}_i(x)$. Then $\tilde{f}_i(x)$ solves
\begin{align}
c\tilde{f}_i'(x)+\sum_{j=1,j\neq i}^{n+1}\lambda_{ij}f_j(x)-(\lambda_i+\delta)\tilde{f}_i(x)=0, x\le b_i^*+\varepsilon,
\end{align}
and $\tilde{f}_i'(b_i^*+\varepsilon)=1$.
Further, $v(x)=\tilde{f}_i(x)-M_i(x)$ solves $v'(x)=(\lambda_i+\delta)v(x)$. Since $v'(b_i^*+\varepsilon)=1-{M_i}'(b_i^*+\varepsilon)>0,$ we conclude that $v(x)>0$ for all $x$.
 Suppose now $J_0\neq i$. If we now follow the strategy with barriers at $\{b_i^*\}_{i=1}^n$, except the first time we reach state $i$, where we use a barrier at $b_i^*+\varepsilon$, we get a higher
value.
In the same way, we find that using a barrier at $b_i^*+\varepsilon$ the first two times state $i$ is visited, a
higher value is reached. Proceeding in the same way, we find that a barrier at $b_i^*+\varepsilon$ instead of $b_i^*$ would
give a larger expected cumulative discounted dividend value. Since we have assumed that the barrier $b_i^*$ is chosen optimally, it contradicts the assumption of $M_i'(x)\le1$. Eventually we see that $M_i'(x)\ge 1$ in an environment of $b_i^*.$  Eventually, we see that ${f_i^*}'(x)$ attains its minimum at $b_i^*$, which shows that ${f_i^*}''(b_i^*)=0.$

Now we summarize the above  two necessary conditions for the optimal phase-wise barrier $\{b_i^*\}_{i=1}^n$:
\begin{enumerate}[(1)]
  \item ${f_i^*}'(x)>1$ on $(0,b_i^*),  i=1,2\ldots, n.$
  \item ${f_i^*}''(b_i^*)=0.$
\end{enumerate}
Different with \cite{MishuraSchmidli2012}, we can not solve the optimal phase-wise barrier strategy's value function directly due to the parameter complexity. Thus, we will treat the case of 2--order and $n$--order separately. In summary,  in the case of 2--order phase-type distributed interclaim times,  we theoretically show the the optimal phase-wise barrier strategy is optimal among all dividend policies. In the case of $n$--order, we use the numerical way to show the same conclusion.
\section{2-order phase-type distributed interclaim times}
\label{sec:2ordercase}
\subsection{The concavity}
\begin{theorem}\label{theoremconcave1}
If $0<b_1^*<b_2^* $, then the optimal phase-wise barrier strategy's value function $f_2^*(x)$ is concave on $[b_1^*,+\infty).$
\end{theorem}

\begin{proof}Recall that $f_i^*(x)$, $i=1,2$, satisfy the equations
\begin{multline}\label{concave3}
  c{f_1^*}'(x)+\lambda_{12}f_2^*(x)-(\lambda_1+\delta)f_1^*(x)+\lambda_{13}\beta e^{-\beta x} \int_0^x e^{\beta y}(\pi_1 f_1^*(y)+\pi_2f_2^*(y))\mathrm{d}y=0, \\ x\in[0,b_1^*];
\end{multline}
\begin{multline}\label{cancave1}
  c{f_2^*}'(x)+\lambda_{21}f_1^*(x)-(\lambda_2+\delta)f_2^*(x)+\lambda_{23}\beta e^{-\beta x} \int_0^x e^{\beta y}(\pi_1 f_1^*(y)+\pi_2f_2^*(y))\mathrm{d}y=0, \\x\in[0,b_2^*].
\end{multline}
\begin{align*}
 f_1^*(x)=f_1^*(b_1^*)+x-b_1^*,\quad & x\in(b_1^*,+\infty).\\
 f_2^*(x)=f_2^*(b_2^*)+x-b_2^*,\quad & x\in(b_2^*,+\infty).
\end{align*}
Taking the derivative in (\ref{cancave1}), we get
\begin{multline}
 c{f_2^*}''(x)+\lambda_{21}{f_1^*}'(x)+(\beta c-(\lambda_2+\delta)){f_2^*}'(x) +\beta(\lambda_{23}\pi_1+\lambda_{21}){f_1^*}(x)\\ +\beta(\lambda_{23}\pi_2-(\lambda_2+\delta))f_2^*(x)=0,\qquad  x\in[0,{b_2^*}).
\end{multline}
Using ${f_1^*}'(x)=1$,
\begin{multline}\label{concave2}
c{f_2^*}'''(x)+(c\beta-(\lambda_2+\delta)){f_2^*}''(x) +\beta(\lambda_{23}\pi_2-(\lambda_2+\delta)){f_2^*}'(x)\\
+\beta(\lambda_{23}\pi_1+\lambda_{21})=0, \qquad x\in[{b_1^*},{b_2^*}).
\end{multline}
Suppose there exists a point $x\in[{b_1^*},{b_2^*})$ such that ${f_2^*}''(x)>0$.
If $c\beta-(\lambda_2+\delta)\le0$, then
\begin{align*}
c{f_2^*}'''(x)\ge -\beta(\lambda_{23}\pi_1+\lambda_{21})+\beta(\lambda_2+\delta-\lambda_{23}\pi_2){f_2^*}'(x).
\end{align*}
Since ${f_2^*}'(x)> 1$ on $[{b_1^*},{b_2^*})$, we see that $c{f_2^*}'''(x)>\beta \delta$.
Thus, ${f_2^*}''(x)>0$ on $(x,{b_2^*}]$, contradicting that ${f_2^*}''({b_2^*})=0$.
If  $c\beta-(\lambda_2+\delta)>0$, taking the derivative in (\ref{concave2}),
\begin{equation*}
c{f_2^*}''''(x)+(c\beta-(\lambda_2+\delta)){f_2^*}'''(x) +\beta(\lambda_{23}\pi_2-(\lambda_2+\delta)){f_2^*}''(x)=0.
\end{equation*}
Because ${f_2^*}''({b_2^*})=0$, the solution of ${f_2^*}''(x)$ is of the form
\begin{equation*}
{f_2^*}''(x)=A(e^{r_1(x-{b_2^*})}-e^{r_2(x-{b_2^*})}),
\end{equation*}
where $r_2<0<r_1$ are the two roots of
\begin{equation*}
cr^2+(c\beta-(\lambda_2+\delta))r+\beta(\lambda_{23}\pi_2-(\lambda_2+\delta))=0.
\end{equation*}
From the assumption that ${f_2^*}''(x)>0$ for some $x\in [{b_1^*},{b_2^*}),$ we conclude that $A<0$.
Thus, ${f_2^*}''(x)>0$ for all $x\in [{b_1^*},{b_2^*})$.
Then we have ${f_2^*}'(x)<1$ for $x\in[{b_1^*},{b_2^*})$, yielding a contradiction.
Until now, we show that for all $x\in[{b_1^*},{b_2^*}),$ ${f_2^*}''(x)\le 0$.
Because ${f_2^*}(x)={f_2^*}({b_2^*})+x-{b_2^*}$ for $x>{b_2^*}$, we find that $f_2^*(x)$ is concave on $[{b_1^*},+\infty)$.
\end{proof}

\subsection{The Optimality}\label{Optimality}
In this section, we will show that the optimal phase-wise barrier dividend  strategy is optimal among all strategies when the interclaim times follow the 2-order phase-type distribution.

\begin{theorem}
If the interclaim times follow a 2-order phase-type distribution, then the optimal phase-wise barrier strategy's value function $f_i^*(x)$ satisfies the following Hamilton-Jacobi-Bellman (HJB) equation
\begin{equation}\label{phasehjb1}
\max\Big\{c{f_i^*}'(x)+\sum_{j=1,j\neq i}^{3}\lambda_{ij}f_j^*(x) -(\lambda_i+\delta){f_i^*}(x),1-{f_i^*}'(x)\Big\}=0,\quad i=1,2,
\end{equation}
where ${f_3^*}(x)=\beta e^{-\beta x}\int_0^x e^{\beta y}(\pi_1{f_1^*}(y) +\pi_2{f_2^*}(y))\mathrm{d}y$.
\end{theorem}

\begin{proof}
Since for all $x>b_i^*$, ${f_i^*}(x)={f_i^*}(b_i^*)+x-b_i^*$, $i=1,2$ and for all $ 0\le x<b_i^*,$ ${f_i^*}(x)$ satisfies \eqref{concave3} and \eqref{cancave1},  we only need to show that
\begin{equation*}
c{f_i^*}'(x)+\sum_{j=1,j\neq i}^{3}\lambda_{ij}f_j^*(x)-(\lambda_i+\delta){f_i^*}(x)\le 0
\end{equation*}
on $[b_i^*,+\infty)$.
We will prove
\begin{equation*}
c{f_1^*}'(x)+\sum_{j=2}^{3}\lambda_{1j}{f_j^*}(x)-(\lambda_1+\delta){f_1^*}(x)\le 0
\end{equation*}
on $[{b_1^*},+\infty)$ first.
Multiplying by $e^{\beta x}$ gives
\begin{equation*}
c{f_1^*}'(x)e^{\beta x}+\lambda_{12}e^{\beta x}{f_2^*}(x)+\lambda_{13}\beta \int_0^xe^{\beta y}(\pi_1{f_1^*}(y)+\pi_2{f_2^*}(y))\mathrm{d}y-(\delta+\lambda_1)e^{\beta x}{f_1^*}(x).
\end{equation*}
Taking the derivative yields
\begin{multline*}
e^{\beta x}\big[\lambda_{12}{f_2^*}'(x)+(\lambda_{12}\beta+\lambda_{13}\pi_2){f_2^*}(x) +(\lambda_{13}\pi_1\beta-\beta(\delta+\lambda_1)){f_1^*}(x)+c\beta-(\delta+\lambda_1)\big]\\
=:e^{\beta x}g_1(x).
\end{multline*}
Noticing that $g_1({b_1^*})=0$, we only need to show that
\begin{multline*}
g_1'(x)=\lambda_{12}{f_2^*}''(x)+\beta(\lambda_{12}+\lambda_{13}\pi_2){f_2^*}'(x)+\beta(\lambda_{13}\pi_1-(\delta+\lambda_1))\le 0,\\ x\in[{b_1^*},+\infty).
\end{multline*}
Since ${f_2^*}(x)$ is concave on $[{b_1^*},+\infty), $ we only need to show that
\begin{equation}\label{phaseopti2}
(\lambda_{12}+\lambda_{13}\pi_2){f_2^*}'(b_1)\le\delta+\lambda_1-\lambda_{13}\pi_1.
\end{equation}
Recall that
\begin{equation}\label{phaseopti1}
c{f_1^*}'(x)+\lambda_{12}f_2^*(x)+\lambda_{13}f_3^*(x)-(\lambda_1+\delta)f_1^*(x)=0,x\in(0,b_1^*].
\end{equation}
Taking the derivative of \eqref{phaseopti1}, we see that
\begin{equation}
\lambda_{12}{f_2^*}'(b_1)+\lambda_{13}{f_3^*}'(b_1)-(\lambda_1+\delta)=0.
\end{equation}
Combining \eqref{phaseopti2} with \eqref{phaseopti1}, we only need to show that
\begin{equation}\label{phaseopti5}
{f_3^*}'(b_1)\ge \pi_1+\pi_2{f_2^*}'(b_1).
\end{equation}
On the other hand,
\begin{equation}\label{phaseopti3}
c{f_2^*}'(x)+\lambda_{21}f_1^*(x)+\lambda_{23}f_3^*(x)-(\lambda_2+\delta)f_2^*(x)=0, x\in(0,b_2].
\end{equation}
Taking the derivative of \eqref{phaseopti3},
\begin{equation*}
c{f_2^*}''(x)+\lambda_{21}{f_1^*}'(x)+\lambda_{23}{f_3^*}'(x)-(\lambda_2+\delta){f_2^*}'(x)=0,
x\in(0,b_2].
\end{equation*}
Since $f_2^*$ is concave on $[b_1,b_2]$ and ${f_1^*}'(b_1)=1$,
\begin{equation*}
\lambda_{21}{f_1^*}'(b_1)+\lambda_{23}{f_3^*}'(b_1)
-(\lambda_2+\delta){f_2^*}'(b_1)\ge 0,
\end{equation*}
After simplification,
\begin{equation}\label{phaseopti4}
{f_3^*}'(b_1)\ge \frac{(\lambda_2+\delta){f_2^*}'(b_1)}{\lambda_{23}}.
\end{equation}
Now we compare the right-hand side of  both \eqref{phaseopti5} and \eqref{phaseopti4}.
Using the fact that ${f_2^*}'(b_1)\ge 1$, we can easily obtain that
\begin{equation*}
{f_3^*}'(b_1)\ge\frac{(\lambda_2+\delta){f_2^*}'(b_1)}{\lambda_{23}}\ge \pi_1+\pi_2{f_2^*}'(b_1).
\end{equation*}
Until now, we show that \eqref{phaseopti5} holds, eventually
\begin{equation*}
c{f_1^*}'(x)+\sum_{j=2}^{3}\lambda_{1j}{f_j^*}(x)-(\lambda_1+\delta){f_1^*}(x)\le 0,\quad x\in[{b_1^*},+\infty).
\end{equation*}

%${f_2^*}'({b_1^*})\le \frac{\delta+\lambda_1-\lambda_{13}\pi_1}{\lambda_{12}+\lambda_{13}\pi_2}.$ Because $\lambda_{23}>\lambda_{13}$, we can prove that $\frac{\lambda_{13}\lambda_{21}+(\lambda_1+\delta)\lambda_{23}}{(\lambda_2+\delta)\lambda_{13}+\lambda_{23}\lambda_{12}}\le  \frac{\delta+\lambda_1-\lambda_{13}\pi_1}{\lambda_{12}+\lambda_{13}\pi_2}.$
%
%Thus, we show that ${f_2^*}'({b_1^*})\le \frac{\delta+\lambda_1-\lambda_{13}\pi_1}{\lambda_{12}+\lambda_{13}\pi_2}$ and eventually we show that

Now we show that
\begin{equation*}
c{f_2^*}'(x)+\lambda_{21}{f_1^*}(x)+\lambda_{23}{f_3^*}(x)-(\lambda_2+\delta){f_2^*}(x)\le 0, \quad x\in[{b_2^*},+\infty).
\end{equation*}
Multiplying by $e^{\beta x}$ gives
\begin{equation*}
ce^{\beta x}{f_2^*}'(x)+\lambda_{21}e^{\beta x}{f_1^*}(x)+\lambda_{23}\beta \int_0^xe^{\beta y}(\pi_1{f_1^*}(y)+\pi_1{f_2^*}(y))dy-(\lambda_2+\delta)e^{\beta x}{f_2^*}(x).
\end{equation*}
Taking the derivative yields
\begin{multline*}
e^{\beta x}(c\beta+\lambda_{21}\beta {f_1^*}(x)+\lambda_{21}+\lambda_{23}\beta (\pi_1{f_1^*}(x)+\pi_2{f_2^*}(x))-(\lambda_2+\delta)\beta {f_2^*}(x)-(\lambda_2+\delta))\\ =:e^{\beta x}g_2(x),
\end{multline*}
here we use that ${f_2^*}''(x)=0$, ${f_1^*}'(x)=1$ on $[{b_2^*},+\infty)$.
Since $g_2({b_2^*})=0$ and
\begin{equation*}
g_2'(x)=\lambda_{21}\beta+\lambda_{23}\beta-(\lambda_2+\delta)\beta<0,
\end{equation*}
we see that $e^{\beta x}g_2(x)\le 0$ on $[{b_2^*},+\infty)$.
Combining this with
\begin{equation*}
c{f_2^*}'({b_1^*})+\lambda_{21}{f_1^*}({b_1^*})+\lambda_{23}{f_3^*}({b_1^*}) -(\lambda_2+\delta){f_2^*}({b_1^*})=0,
\end{equation*}
we see that
\begin{equation*}
c{f_2^*}'(x)+\lambda_{21}{f_1^*}(x)+\lambda_{23}{f_3^*}(x)-(\lambda_2+\delta){f_2^*}(x)\le 0, \quad x\in[{b_2^*},+\infty).
\end{equation*}
Until now, we show that the optimal phase-wise barrier strategy's value function satisfies the HJB equation.
\end{proof}

\begin{theorem}
For exponential distributed claim sizes and 2-order phase-type distributed interclaim times, the optimal phase-wise barrier strategy is the optimal policy among all dividend strategies.
\end{theorem}

\begin{proof}
Clearly $f_i^*(x)\le V_i(x)$, where $V_i(x)$ is defined in \eqref{phasevaluefunction}.
Let $L_t$ be an arbitrary dividend process and $X_t^L$ denote the corresponding surplus process. Then by It\^o's formula,
\begin{align*}
f_{J_t}^*(X_{\tau\wedge t}^L)e^{-\delta (\tau\wedge t)}=&f_i^*(x)+\int_0^{\tau\wedge t}e^{-\delta s}[-\delta f_{J_s}^*(X_s^L)+\mathscr{L}f_{J_s}^*(X_s^L)]\mathrm{d}s \\
&-\int_0^{\tau\wedge t}{f_{J_s}^*}'(X_s^L)e^{-\delta s}\mathrm{d}L_s,
\end{align*}
where $\tau$ is the ruin time and
\begin{equation*}
\mathscr{L}f_i^*(x)=c{f_i^*}'(x)+\sum_{j=1,j\neq i }^{3}\lambda_{ij}{f_j^*}(x)-\lambda_{i}f_i^*(x)
\end{equation*}
is the infinitesimal generator of the process.
From \eqref{phasehjb1} we conclude that
\begin{align*}
\mathbb{E}\left[f_{J_t}^*(X_{\tau\wedge t}^L)e^{-\delta (\tau\wedge t)}\right]
&\le
f_i^*(x)-\mathbb{E}\left[\int_0^{\tau\wedge t}{f_{J_s}^*}'(X_s^L)e^{-\delta s}\mathrm{d}L_s \Big|J_0=i,X_{0-}^L=x\right]\\
&\le f_i^*(x)-\mathbb{E}\left[\int_0^{\tau\wedge t}e^{-\delta s}\mathrm{d}L_s \Big|J_0=i,X_{0-}^L=x\right].
\end{align*}

Letting $t\rightarrow \infty$ gives
\[f_i^*(x)\ge \mathbb{E}\left[\int_0^{\tau}e^{-\delta s}\mathrm{d}L_s \Big|J_0=i,X_{0-}^L=x\right].\]
Since the strategy $L_t$ is arbitrary,  $f_i^*(x)\ge  V_i(x)$. Now we complete the proof.
\end{proof}
Until now, we see that the 2--layer barrier dividend strategy is optimal when the interclaim times follow 2--order phase--type distribution. In what follows, we will find a way to compare the those two barriers.
\subsection{The comparison of two optimal barriers}
\begin{theorem}\label{theoremlambda}
If $0<{b_1^*}<{b_2^*} $, then $\lambda_{23}>\lambda_{13}$.
\end{theorem}

\begin{proof}
Denote ${f_3^*}(x)=\beta e^{-\beta x}\int_0^xe^{\beta y}(\sum_{i=1}^2\pi_i f_i^*(y))\mathrm{d}y$.
From (\ref{cancave1}), we see that
\begin{equation*}
c{f_2^*}''(x)+\lambda_{21}-(\lambda_2+\delta){f_2^*}'(x)+\lambda_{23}{f_3^*}'(x)=0, x\in[{b_1^*},{b_2^*}].
\end{equation*}
Combining the concavity of ${f_2^*}(x)$ on $[{b_1^*},{b_2^*}]$, we see
\begin{equation}\label{necessary1}
\lambda_{21}-(\lambda_2+\delta){f_2^*}'({b_1^*})+\lambda_{23}{f_3^*}'({b_1^*})\ge 0.
\end{equation}
On the other hand, from (\ref{concave3}),
\begin{equation}\label{necessary2}
\lambda_{12}{f_2^*}'({b_1^*})+\lambda_{13}{f_3^*}'({b_1^*})-(\lambda_1+\delta)=0.
\end{equation}
Combining (\ref{necessary1}) with (\ref{necessary2}), we see that
\begin{equation}\label{concave4}
{f_2^*}'({b_1^*})\le \frac{\lambda_{13}\lambda_{21}+(\lambda_1+\delta)\lambda_{23}}{(\lambda_2+\delta)\lambda_{13}+\lambda_{23}\lambda_{12}}.
\end{equation}
Since ${f_2^*}'({b_1^*})>1$, we know $\frac{\lambda_{13}\lambda_{21}+(\lambda_1+\delta)\lambda_{23}}{(\lambda_2+\delta)\lambda_{13}+\lambda_{23}\lambda_{12}}>1.$
After simplification, we see that $\lambda_{23}>\lambda_{13}$.
\end{proof}
The above theorem shows that when the interclaim times are 2--order phase-type distributed, the phase with the  higher barrier must be the phase with the higher intensity to claim.

Now we show some numerical examples for the case of 2--order phase-type distributed interclaim times. The algorithm is shown in the next section.
\begin{example}\label{exam1}
Let the interclaim times follow a 2-order phase-type distribution.
The state $3$ is the absorption state of the Markov process $\{J_t\}$, which means, once the Markov process enters the state 3, the claim occurs. The intensity matrix $\mathbf{\Lambda}$ is
\begin{equation*}
\left(
  \begin{array}{ccc}
      -\lambda_{1} & \lambda_{12} &\lambda_{13}\\
        \lambda_{21} & -\lambda_{2}&\lambda_{23} \\
  \end{array}
\right)=\left(\begin{array}{ccc}
 -10 & 5&5 \\ 4 & -12 & 8  \\
 \end{array}\right).
\end{equation*}

And $c=15$, $\delta=0.1$, $\beta=1$, $(\pi_1,\pi_2)=(0.4,0.6)$.
Table \ref{tab:2order1} shows  $b_1^*<b_2^*$.
\begin{table}[h]
  \centering
  \begin{tabular}{lcc}
  \hline
Phase $i$ & 1 & 2 \\\hline
%  Expected time of the next claim $\mathcal{T}_i$ & 0.17 & 0.14 \\
  Optimal phase-wise barrier $b_i^*$& 11.779 & 12.219 \\
  \hline
  \end{tabular}
  \caption{$b^*_i$ for a 2-order case with $\lambda_{23}>\lambda_{13}$.}
  \label{tab:2order1}
\end{table}
The functions $f_1^*(x)$ and $f_2^*(x)$ are shown in Figure \ref{fig:2order1}.
%\begin{table}[h]
%  \centering
%  \begin{tabular}{lcc}
%  \hline
%  State $i$ & 1 & 2 \\\hline
%  Expected time of the next claim $\mathcal{T}_i$ & 0.17 & 0.14 \\
%  Optimal phase-wise barrier $b_i^*$& 11.779 & 12.219 \\
%  \hline
%  \end{tabular}
%  \caption{$\mathcal{T}_i$ and $b^*_i$ for a 2-order case with $\mathcal{T}_1>\mathcal{T}_2$.}
%  \label{tab:2order1}
%\end{table}
%\begin{figure}[htb]
% \center{\includegraphics[width=0.6\textwidth]{2dimensional1.png}}
% \caption{Functions $f_i^*(x)-x$ for $\mathcal{T}_1>\mathcal{T}_2$.}\label{2d}
% \end{figure}
\end{example}

Another numerical experiment for 2-order case is shown next for what happens if the intensity to the next claim  is the same.

\begin{example} \label{exam2}
Let the interclaim times follow a 2-order phase-type distribution, and $c=15$, $\delta=0.1$, $\beta=1$, $(\pi_1,\pi_2)=(0.4,0.6)$. The intensity matrix $\mathbf{\Lambda}$ is
\begin{equation*}
\left(\begin{array}{ccc}
    -\lambda_{1} &\lambda_{12} &\lambda_{13}\\
    \lambda_{21} & -\lambda_{2} &\lambda_{23}\\
\end{array}\right)
=\left(\begin{array}{ccc}
 -8 & 3 &5 \\ 1 & -6  & 5\\
 \end{array}\right),
\end{equation*}
\begin{table}[h]
  \centering
  \begin{tabular}{lcc}
  \hline
  State $i$ & 1 & 2 \\ \hline
 % Expected time of the next claim $\mathcal{T}_i$ & 0.2 & 0.2 \\
  Optimal phase-wise barrier $b_i^*$ &10.738  & 10.738 \\
  \hline
  \end{tabular}
  \caption{$b^*_i$ for a 2-order case with $\lambda_{13}=\lambda_{23}$.}
  \label{tab:2order2}
\end{table}
From Table \ref{tab:2order2}, we see that  the barriers are  the same.
The equality of $f_1^*(x)$ and $f_2^*(x)$ is shown in the Figure \ref{fig:2order2}.
\begin{figure}[htb]
\centering
\begin{minipage}[t]{0.48\textwidth}
 \center{\includegraphics[width=\textwidth]{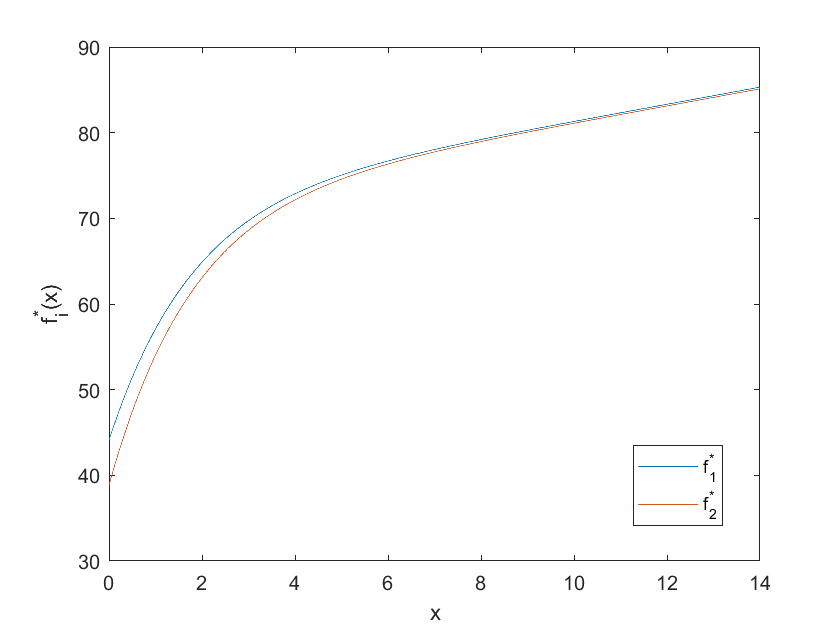}}
 \caption{$f_i^*(x)$ for $\lambda_{23}>\lambda_{13}$.}\label{fig:2order1}
\end{minipage}
\begin{minipage}[t]{0.48\textwidth}
 \center{\includegraphics[width=\textwidth]{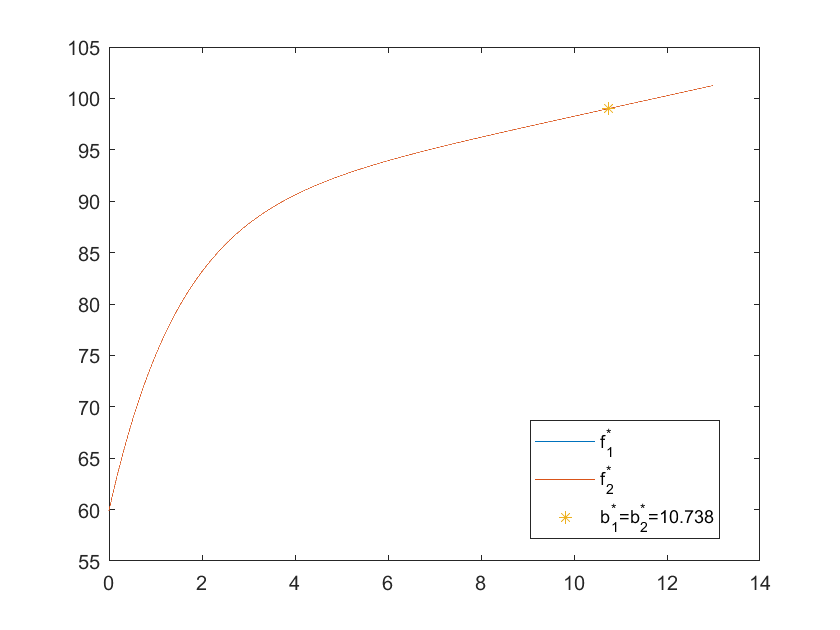}}
 \caption{$f_i^*(x)$ for $\lambda_{13}=\lambda_{23}$.}\label{fig:2order2}
\end{minipage}
\end{figure}
%We see that in this case, $f_1^*(x)=f_2^*(x)$.
%\begin{figure}[htb]
%  \center{\includegraphics[width=0.6\textwidth]{2dimensional2.png}}
%  \caption{Functions $f_i^*(x)-x$ for $\mathcal{T}_1>\mathcal{T}_2$.}
%  \label{2d2}
%\end{figure}
\end{example}

\begin{remark}
When the interclaim times are 2-order phase-type distributed, we can also compare the size of the barrier by  comparing different phases' the expected time to the next claim.
 Recall that for a given phase $i$, the expected time of the next claim $\mathcal{T}_i$ is, confer \cite{AsmussenAlbrecher2010} (Page 256, Theorem 1.5-(d)),
\begin{equation}
\mathcal{T}_i=\alpha_i \mathbf{T}^{-1}\mathbf{e},
\end{equation}
where $\alpha_i=(0,0,\cdots,0, -1,0,\cdots,0)$ is an $n$-dimensional vector with $-1$ being the $i$th term, $\mathbf{T}$ is the subintensity matrix of $\mathbf{\Lambda}$ restricted to $E$ and $\mathbf{e}$ is the column vector with all components equal to one.

 If $\lambda_{23}\ge \lambda_{13}$, then
 \begin{align*}
\mathcal{T}_1-\mathcal{T}_2&=
\Big(\begin{array}{cc}
            -1 & 1 \\
\end{array}\Big)\cdot\mathbf{T}^{-1}\cdot
\left(\begin{array}{c}
             1 \\ 1
\end{array}\right)
\\&
=\frac{1}{\lambda_1\lambda_2-\lambda_{12}\lambda_{21}}
\Big(\begin{array}{cc}
            -1 & 1 \\
\end{array}\Big)\cdot
\left(\begin{array}{c}
   -\lambda_2-\lambda_{12} \\
   -\lambda_{21}-\lambda_1
\end{array}\right)\\
&=\frac{1}{\lambda_1\lambda_2-\lambda_{12}\lambda_{21}}(\lambda_{23}-\lambda_{13})\ge 0.
\end{align*}
Eventually, combining this with Theorem \ref{theoremlambda}, we get that if $0<{b_1^*}<{b_2^*} $, then $\mathcal{T}_1\ge \mathcal{T}_2$, which shows that for different phases, except by comparing the intensity to the claim to compare the size of barriers,   we can also compare the expected time to the next claim to obtain which barrier is larger.
\end{remark}

\section{n-order ($n\ge 3$) phase-type distributed interclaim times}
\label{sec:nordercase}
Since there are too many parameters in the case of $n$--order $(n\ge 3)$, a similar calculation like Theorem \ref{theoremconcave1} is almost impossible. Thus, we numerically explore the optimal policy. We present an algorithm first and then we show that the convergence of the algorithm. By the numerical exploration, we get that the optimal phase-wise barrier strategy is optimal among all admissible dividend policies.
\subsection{Algorithm}
We use the iteration algorithm which is similar with the ones used in \cite{Albrecher2017} and \cite{Liu2020Dividend}.

\begin{enumerate}[Step 1.]
  \item Set $k=0$ and initialize with $V^{(0)}_i\equiv 0$, $i=1,2,\cdots,n+1$.
  \item Set $k=k+1$. For given $V_i^{(k-1)}$, we look for the optimal dividend strategy prior to $\tau_1=\inf\{t\geq 0:\bar{J}_t\neq J_0\}$ with terminal cost,
      \begin{equation}\label{eq:ValueFunctionK}
        V_i^{(k)}(x)=\sup_{D\in\mathscr{U}_{ad}}\mathbb{E}\Big[\int_{0-}^{\tau_1-} e^{-\delta t}\mathrm{d}D_t +e^{-\delta\tau_1}V_{\bar{J}_{\tau_1}}^{(k-1)}(X^D_{\tau_1}) \Big|J_0=i,X^D_{0-}=x\Big]
      \end{equation}
      for $i=1,2,\cdots,n$.
      The corresponding HJB equation is
      \begin{equation}%\label{}
        \max\Big\{c{V_i^{(k)}}'(x)-(\lambda_i+\delta){V_i^{(k)}}(x)+\sum_{j=1,j\neq i}^{n+1}\lambda_{ij}V_j^{(k-1)}(x) ,1-{V_i^{(k)}}'(x)\Big\}=0.
      \end{equation}
  \item If $\max_i\sup_{x\geq 0}|V_i^{(k)}(x)-V_i^{(k-1)}(x)|<\textrm{tolerance}$ for $k>1$ stop; otherwise, go to Step 2.
\end{enumerate}

Actually, in Step 2, we find that the optimal strategy is a phase-wise barrier strategy with barriers
      \begin{equation}\label{eq:barrier}
        b_i^{(k)}=\mathop{\arg\max}_{x\ge 0}\left\{\frac{c+\sum_{j=1,j\neq i}^{n+1}\lambda_{ij}V_j^{(k-1)}(x)}{\lambda_i+\delta}-x\right\},\quad i=1,2,\cdots,n.
      \end{equation}
      And the corresponding functions $V_i^{(k)}$ are obtained by
      \begin{equation}
        V_i^{(k)}(x)=
        \begin{cases}
          \displaystyle e^{-\frac{\lambda_i+\delta}{c}(b_i^{(k)}-x)} V_i^{(k)}(b_i^{(k)}) &  \\
          \displaystyle \quad + \frac{1}{c}\int_x^{b_i^{(k)}}e^{-\frac{\lambda_i+\delta}{c}(u-x)} \Big(\sum_{\substack{j=1\\j\neq i}}^{n+1} \lambda_{ij}V_j^{(k-1)}(u)\Big)\mathrm{d}u, & x<b_i^{(k)}; \\
          \displaystyle\frac{c+\sum_{j=1,j\neq i}^{n+1}\lambda_{ij}V_j^{(k-1)}(x)}{\lambda_i+\delta}, & x=b_i^{(k)}; \\
          V_i^{(k)}(b_i^{(k)})+x-b_i^{(k)}, & x>b_i^{(k)}, \\
        \end{cases}
      \end{equation}
      for $i=1,2,\cdots,n$, and
      \begin{equation}
        V_{n+1}^{(k)}(x)=\beta e^{-\beta x}\int_0^xe^{\beta y}\left(\sum_{i=1}^n\pi_iV_i^{(k)}(y)\right)\mathrm{d}y.
      \end{equation}
After we present the algorithm, we show the convergence of the algorithm. The proof contains two parts, the first part (more specifically, Proposition \ref{valuefunctionconvergence})proves that $\{V_i^{(k)}\}_{i=1}^n$  converges to the optimal value function $\{V_i\}_{i=1}^n$, the second part (Proposition \ref{convegencebarrier}) proves that iteration strategy $\{b_i^{(k)}\}_{i=1}^n$ convergence to the optimal policy $\{b_i^*\}_{i=1}^n$.
\subsection{The convergence of the algorithm}
\begin{lemma}\label{lem:Vk.increasing}
  We have $V_i^{(0)}\leq V_i^{(1)}\leq \cdots \leq V_i^{(k)}\leq \cdots \leq V_i$ for $i=1,2,\cdots,n$.
\end{lemma}

\begin{proof}
  We prove the result by induction.

  (i) It follows from $V_i^{(0)}\equiv 0$ that the optimal strategy is barrier strategy with $b_i^{(1)}=0$, i.e.,
  \begin{equation}
  V_i^{(1)}(x)=\sup_{D\in\mathscr{U}_{ad}}\mathbb{E}\Big[\int_{0-}^{\tau_1-} e^{-\delta t}\mathrm{d}D_t\Big|J_0=i,X^D_{0-}=x\Big] =x+\frac{c}{\lambda_i+\delta}.
  \end{equation}
  We have $V_i^{(0)}\leq V_i^{(1)}$.

  (ii) Assume that $V_i^{(k-2)}\leq V_i^{(k-1)}$. By (\ref{eq:ValueFunctionK}), we have
  \begin{align*}
    V_i^{(k)}(x) \geq & \sup_{D\in\mathscr{U}_{ad}}\mathbb{E}\Big[\int_{0-}^{\tau_1-} e^{-\delta t}\mathrm{d}D_t +e^{-\delta\tau_1}V_{\bar{J}_{\tau_1}}^{(k-2)}(X^D_{\tau_1})\Big|J_0=i,X^D_{0-}=x\Big] \\
    =&V_i^{(k-1)}(x).
  \end{align*}
\end{proof}

\begin{proposition}\label{valuefunctionconvergence}
  It holds that $V_i^{(k)}\to V_i$ as $k\to\infty$ for $i=1,2,\cdots,n$.
\end{proposition}

\begin{proof}
  For any admissible strategy $D\in\mathscr{U}_{ad}$, we have
  \begin{equation*}
    D_t\leq x+ct.
  \end{equation*}
  Then
  \begin{equation*}
    V_i(x)=\sup_{D\in\mathscr{U}_{ad}}V_i^D(x)\leq x+\int_0^\infty ce^{-\delta t}\mathrm{d}t=x+\frac{c}{\delta}.
  \end{equation*}
  So there exists $T>0$ such that
  \begin{equation}\label{ieq:1}
    e^{-\delta t}V_i(x+ct)<\frac{1}{3}\varepsilon\quad\hbox{for}\quad t>T.
  \end{equation}
  Let us define $K=\min_iV_i(x+pT)$ and take $N>0$ such that
  \begin{equation}\label{ieq:2}
    \mathbb{P}\{\tau_N\geq T\}\geq 1-\frac{\varepsilon}{3K}.
  \end{equation}
  There exists an admissible strategy $D^\varepsilon\in\mathscr{U}_{ad}$ such that
  \begin{equation}\label{ieq:3}
    V_i(x)-V_i^{D^\varepsilon}(x)<\frac{1}{3}\varepsilon.
  \end{equation}

 Denote  $\tau_N$ the $N^{th}$ jump time of the Markov process $\{J_t\}$ (The claim counts as one jump). We define the strategy $D^N$ as $D^N_t=D^\varepsilon_t$ for $t\leq\tau_N\land\tau$, and $D^N_t=D^\varepsilon_{\tau_N}$ for $t>\tau_N$ if $\tau_N<\tau$.
  From (\ref{ieq:1}) and (\ref{ieq:2}), we have
  \begin{align*}
    &V_i^{D^\varepsilon}(x)-V_i^{D^N}(x) \\
    =&\mathbb{E}_{ix}\Big[\int_{\tau_N\land\tau-}^{\tau-} e^{-\delta t} \mathrm{d}D^\varepsilon_t\Big]
    \leq \mathbb{E}_{ix}\Big[e^{-\delta(\tau_N\land\tau)}V_{\bar{J}_{\tau_N\land\tau}} (X^{D^\varepsilon}_{\tau_N\land\tau})\Big] \\
    \leq & \mathbb{E}_{ix}\Big[\mathbf{1}_{\{\tau_N\land\tau\geq T\}} e^{-\delta(\tau_N\land\tau)}V_{\bar{J}_{\tau_N\land\tau}} (x+c(\tau_N\land\tau))\Big] \\
    &+\mathbb{E}_{ix}\Big[\mathbf{1}_{\{\tau_N\leq T\}} V_{\bar{J}_{\tau_N\land\tau}} (x+c(\tau_N\land\tau))\Big] \\
    \leq & \mathbb{E}_{ix}\Big[\mathbf{1}_{\{\tau_N\land\tau\geq T\}} e^{-\delta(\tau_N\land\tau)}V_{\bar{J}_{\tau_N\land\tau}} (x+c(\tau_N\land\tau))\Big] + K\mathbb{P}\{\tau_N < T\} \\
    < & \frac{2}{3}\varepsilon.
  \end{align*}
  The we obtain from (\ref{ieq:3}),
  \begin{equation}\label{ieq:4}
    V_i(x)\leq V_i^{D^\varepsilon}(x)+\frac{1}{3}\varepsilon <V_i^{D^N}(x)+\varepsilon.
  \end{equation}

  If we define a sequence of strategies $D^1, D^2, \cdots, D^{N-1}$ as $D^k_t=D^N_{\tau_{N-k}+t}$, we can see that
  \begin{equation*}
    V_i^{D^k}(x)=\mathbb{E}_{ix}\Big[\int_{0-}^{\tau_1-} e^{-\delta t}\mathrm{d}D^k_t +e^{-\delta\tau_1}V_{\bar{J}_{\tau_1}}^{D^{k-1}}(X^{D^k}_{\tau_1})\Big] \quad\hbox{for}\quad k=1,\cdots,N
  \end{equation*}
  with $V_i^{D^0}=V^{(0)}\equiv 0$. These imply that $V_i^{D^1}(x)\leq V_i^{(1)}(x)$.
  Assume that $V_i^{D^{k-1}}(x)\leq V_i^{(k-1)}(x)$. We have
  \begin{align*}
    V_i^{D^k}(x)=&\mathbb{E}_{ix}\Big[\int_{0-}^{\tau_1-} e^{-\delta t}\mathrm{d}D^k_t +e^{-\delta\tau_1}V_{\bar{J}_{\tau_1}}^{D^{k-1}}(X^{D^k}_{\tau_1})\Big] \\
    \leq & \mathbb{E}_{ix}\Big[\int_{0-}^{\tau_1-} e^{-\delta t}\mathrm{d}D^k_t +e^{-\delta\tau_1}V_{\bar{J}_{\tau_1}}^{(k-1)}(X^{D^k}_{\tau_1})\Big]
    \leq V_i^{(k)}(x).
  \end{align*}
  Hence we obtain
  \begin{equation*}
    V_i^{D^N}(x)\leq V_i^{(N)}(x).
  \end{equation*}
  Following from (\ref{ieq:4}) and Lemma \ref{lem:Vk.increasing}, we have
  \begin{equation*}
    V_i(x)<V_i^{(N)}(x)+\varepsilon\leq V_i^{(k)}(x)+\varepsilon \quad\hbox{for}\quad k>N.
  \end{equation*}
  On the other hand, we have $V_i(x)\geq V_i^{(k)}(x)$.
  This completes the proof.
\end{proof}

\begin{proposition}\label{convegencebarrier}
  The optimal value function $V_i$ is the minimal nonnegative solution of the equation
  \begin{equation}\label{eq:V.funEq}
    V_i(x)=\sup_{D\in\mathscr{U}_{ad}}\mathbb{E}\Big[\int_{0-}^{\tau_1-} e^{-\delta t}\mathrm{d}D_t +e^{-\delta\tau_1}V_{\bar{J}_{\tau_1}}(X^D_{\tau_1}) \Big|J_0=i,X^D_{0-}=x\Big].
  \end{equation}
  If there exists an optimal phase-wise barrier strategy with $\{b_i^*\}_{i=1}^n$ prior to $\tau_1$ for the problem (\ref{eq:ValueFunctionK}) after replacing $V_i^{(k-1)}$ by $V_i$, then the phase-wise barrier strategy with $\{b_i^*\}_{i=1}^n$ is optimal for the optimization problem (\ref{phasevaluefunction}).
\end{proposition}

\begin{proof}
  The dynamic programming principle implies that the optimal value function $V_i$ is a nonnegative solution of (\ref{eq:V.funEq}).
  On the other hand, let $W_i$ be a nonnegative solution of (\ref{eq:V.funEq}), the $W_i\geq V_i^{(0)}\equiv 0$. Assume that $W_i\geq V_i^{(k-1)}$. We have
  \begin{align*}
    W_i(x)=&\sup_{D\in\mathscr{U}_{ad}}\mathbb{E}_{ix}\Big[\int_{0-}^{\tau_1-} e^{-\delta t}\mathrm{d}D_t +e^{-\delta\tau_1}W_{\bar{J}_{\tau_1}}(X^D_{\tau_1})\Big] \\
    \geq& \sup_{D\in\mathscr{U}_{ad}}\mathbb{E}_{ix}\Big[\int_{0-}^{\tau_1-} e^{-\delta t}\mathrm{d}D_t +e^{-\delta\tau_1}V_{\bar{J}_{\tau_1}}^{(k-1)}(X^D_{\tau_1})\Big]
    = V_i^{(k)}(x).
  \end{align*}
  Hence, $W_i\geq V_i^{(k)}$ for all $k\geq 0$.
  It follows from $V_i^{(k)}\to V_i$ that $W_i\geq V_i$.
  These conclude that $V_i$ is the minimal nonnegative solution of (\ref{eq:V.funEq}).

  Suppose that the phase-wise barrier strategy $D^*$ with $\{b_i^*\}_{i=1}^n$ is the optimal strategy prior to $\tau_1$. Then
  \begin{equation*}
    V_i(x)=\mathbb{E}_{ix}\Big[\int_{0-}^{\tau_1-} e^{-\delta t}\mathrm{d}D^*_t +e^{-\delta\tau_1}V_{\bar{J}_{\tau_1}}(X^{D^*}_{\tau_1})\Big].
  \end{equation*}
  We have by induction that
  \begin{equation*}
    V_i(x)=\mathbb{E}_{ix}\Big[\int_{0-}^{\tau_n-} e^{-\delta t}\mathrm{d}D^*_t +e^{-\delta\tau_n}V_{\bar{J}_{\tau_n}}(X^{D^*}_{\tau_n})\Big]
  \end{equation*}
  for all $n\geq 1$.
  In view of that $\mathbb{E}_{ix}[e^{-\delta\tau_n}V_{\bar{J}_{\tau_n}}]\to 0$ as $n\to\infty$, and letting $n\to\infty$, we have
  \begin{equation*}
    V_i(x)=\mathbb{E}_{ix}\Big[\int_{0-}^{\tau-} e^{-\delta t}\mathrm{d}D^*_t\Big].
  \end{equation*}
\end{proof}

\subsection{Examples}
In what follows, we show some examples.

%\begin{table}[h]
%  \centering
%  \begin{tabular}{lcc}
%  \hline
%  State $i$ & 1 & 2 \\\hline
%  Expected time of the next claim $\mathcal{T}_i$ & 0.17 & 0.14 \\
%  Optimal phase-wise barrier $b_i^*$& 11.779 & 12.219 \\
%  \hline
%  \end{tabular}
%  \caption{$\mathcal{T}_i$ and $b^*_i$ for a 2-order case with $\mathcal{T}_1>\mathcal{T}_2$.}
%  \label{tab:2order1}
%\end{table}
%\begin{figure}[htb]
% \center{\includegraphics[width=0.6\textwidth]{2dimensional1.png}}
% \caption{Functions $f_i^*(x)-x$ for $\mathcal{T}_1>\mathcal{T}_2$.}\label{2d}
% \end{figure}

\begin{example} \label{exam3}
Let the interclaims follow a 3-order phase-type distribution.
The intensity matrix $\mathbf{T}$ is
\begin{equation*}\nonumber
\left(\begin{array}{cccc}
    -\lambda_{1} & \lambda_{12} & \lambda_{13} &\lambda_{14}  \\
    \lambda_{21} & -\lambda_{2} & \lambda_{23} &\lambda_{24} \\
    \lambda_{31} & \lambda_{32} & -\lambda_{3}  &\lambda_{34}\\
\end{array}\right)
=\left(\begin{array}{cccc}
 -10 & 5 & 2 & 3\\ 2 & -12 & 4  & 6\\ 2 & 4 & -8 & 2\\
\end{array}\right),
\end{equation*}
$c=21.4$, $\delta=0.1$, $\beta=1$, $(\pi_1,\pi_2,\pi_3)=(0.2,0.3,0.5)$.
%$\lambda_1=10$, $\lambda_2=12$, $\lambda_3=8$, $\pi_1=0.2$, $\pi_2=0.3$, $\pi_3=0.5$.
Table \ref{tab:3order1} shows that phase 2 has the highest barrier.
\begin{table}[h]
  \centering
  \begin{tabular}{lccc}
    \hline
    Phase $i$ & 1 & 2 & 3 \\\hline
    The intenisty to the next claim $\lambda_{i4}$& 3&6 &2\\\hline
  %  Expected time of the next claim $\mathcal{T}_i$ & 0.27922 & 0.23376& 0.31168\\
    Optimal phase-wise barriers $b_i^*$ & 9.61 & 10.26 & 9.27 \\ \hline
  \end{tabular}
  \caption{ $b_i^*$ for a 3-order case with $(\pi_1,\pi_2,\pi_3)=(0.2,0.3,0.5)$.}
  \label{tab:3order1}
\end{table}
The functions $f_i^*(x)$ are shown in Figure \ref{fig:3order}.
\begin{figure}[htb]
\center{\includegraphics[width=0.7\textwidth]{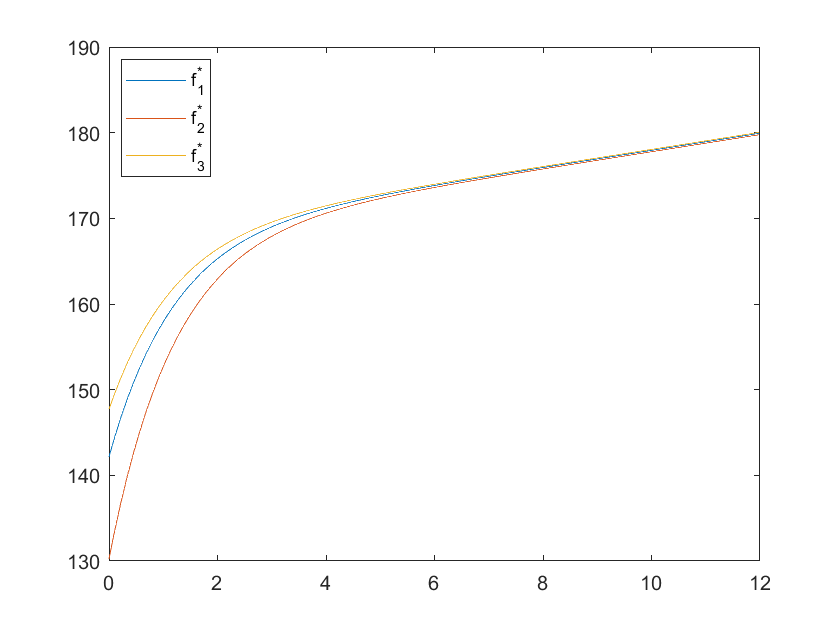}}
  \caption{$f_i^*(x)$ for Example \ref{exam3}.}
  \label{fig:3order}
\end{figure}
%\begin{figure}[htbp]
%  \center{\includegraphics[width=0.6\textwidth]{3dimensional.png}}
%  \caption{Functions $f_i^*(x)-x$ for $\mathcal{T}_3>\mathcal{T}_1>\mathcal{T}_2$.}
%  \label{3d}
%\end{figure}

If the initial probabilities $(\pi_1,\pi_2,\pi_3)=(0.2,0.3,0.5)$ are replaced by $(0.1,0.1,0.8)$,  the new barriers are shown in Table \ref{tab:3order2}.
As we can see, phase 2 still has the highest barrier.
\begin{table}[h]
  \centering
  \begin{tabular}{lccc}
    \hline
    Phase $i$ & 1 & 2 & 3 \\\hline
  %  Expected time of the next claim $\mathcal{T}_i$ & 0.27922 & 0.23376& 0.31168\\
    Optimal phase-wise barrier $b_i^*$ & 9.39 & 10.03 & 9.05 \\ \hline
  \end{tabular}
  \caption{$b_i^*$ for a 3-order case with $(\pi_1,\pi_2,\pi_3)=(0.1,0.1,0.8)$.}
  \label{tab:3order2}
\end{table}
\end{example}

%Now we show a 4-order case.
\begin{example}\label{exam4}
Let the interclaim times follow a 4-order phase-type distribution, and $c=25$, $\delta=0.1$, $\beta=1$, $(\pi_1,\pi_2,\pi_3,\pi_4)=(0.5,0.2,0.2,0.1)$.
 The intensity matrix $\mathbf{T}$ is
\begin{equation*}
\left(\begin{array}{ccccc}
    -\lambda_1 & \lambda_{12} & \lambda_{13} & \lambda_{14} &\lambda_{15} \\
    \lambda_{21} & -\lambda_2 & \lambda_{23} & \lambda_{24} &\lambda_{25} \\
    \lambda_{31} & \lambda_{32} & -\lambda_{3} & \lambda_{34}&\lambda_{35}\\
    \lambda_{41} & \lambda_{42} & \lambda_{43} & -\lambda_{4} &\lambda_{45} \\
\end{array}\right)=
\left(\begin{array}{ccccc}
    -10 & 5 & 2 & 1 &2 \\
    3 & -14 & 4 & 3 &4 \\
    2 & 2 & -12 & 7 &1 \\
    2 & 3 & 1 & -6 &0 \\
\end{array}\right).
\end{equation*}
The optimal phase-wise barriers are shown in Table \ref{tab:4order}.
In this case, we can see that phase 2 has the highest barrier.
\begin{table}[h]
  \centering
\begin{tabular}{lcccc}
  \hline
  Phase $i$ & 1 & 2 & 3 & 4 \\\hline
  The intensity to the next claim $\lambda_{i5}$&2&4&1&0\\\hline
 % Expected  time of the next claim $\mathcal{T}_i$ & 0.440876 & 0.399397 & 0.483191 & 0.445392 \\
  Optimal phase-wise barrier $b_i^*$ & 8.907 & 9.554 & 8.274 & 7.106 \\
  \hline
\end{tabular}
  \caption{ $b_i^*$ for a 4-order case}
  \label{tab:4order}
\end{table}
\begin{figure}[htb]
\center{\includegraphics[width=0.7\textwidth]{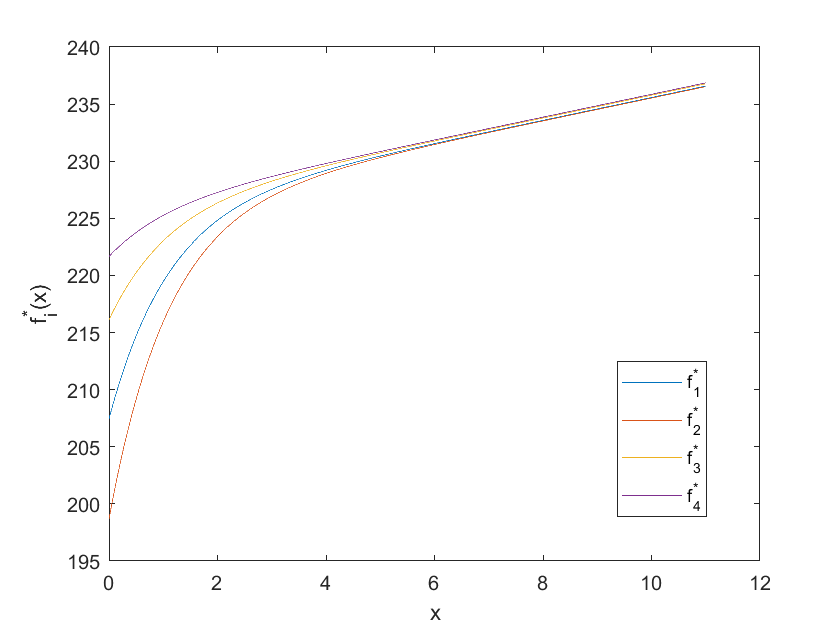}}
  \caption{$f_i^*(x)$ for Example \ref{exam4}.}
  \label{fig:4ordernoequal}
\end{figure}

\end{example}
\begin{example}\label{fourorder2same}
Let us try another example with phase 1 and phase 4 having the same highest intensity. $c=25$, $\delta=0.1$, $\beta=1$, $(\pi_1,\pi_2,\pi_3,\pi_4)=(0.5,0.2,0.2,0.1)$.
The intensity matrix $\mathbf{T}$ is
\begin{equation*}
\left(\begin{array}{ccccc}
    -\lambda_1 & \lambda_{12} & \lambda_{13} & \lambda_{14} &\lambda_{15} \\
    \lambda_{21} & -\lambda_2 & \lambda_{23} & \lambda_{24} &\lambda_{25} \\
    \lambda_{31} & \lambda_{32} & -\lambda_{3} & \lambda_{34}&\lambda_{35}\\
    \lambda_{41} & \lambda_{42} & \lambda_{43} & -\lambda_{4} &\lambda_{45} \\
\end{array}\right)=
\left(\begin{array}{ccccc}
    -16 & 7 & 3 & 1 &5 \\
    4 & -8 & 1 & 2 &1 \\
    0 & 1 & -4 & 1 &2 \\
    0 & 0 & 0 & -5 &5 \\
\end{array}\right).
\end{equation*}
The optimal phase-wise barriers are shown in Table \ref{tab:4order2}. In this example, the phase 1 and phase 4 have the same highest intensity to the next claim and phase 1, 4 have the same highest barrier.
\begin{table}[h]
  \centering
\begin{tabular}{lcccc}
  \hline
  State  $i$ & 1 & 2 & 3 & 4 \\\hline
  The intensity to the next claim $\lambda_{i5}$ &5&1&2&5\\\hline
 % Expected  time of the next claim $\mathcal{T}_i$ & 0.440876 & 0.399397 & 0.483191 & 0.445392 \\
  Optimal phase-wise barrier $b_i^*$ & 10.109 & 8.805 & 9.205 & 10.109 \\
  \hline
\end{tabular}
  \caption{ $b_i^*$ with two same highest intensity  }
  \label{tab:4order2}
\end{table}

\begin{figure}[htb]
\center{\includegraphics[width=0.7\textwidth]{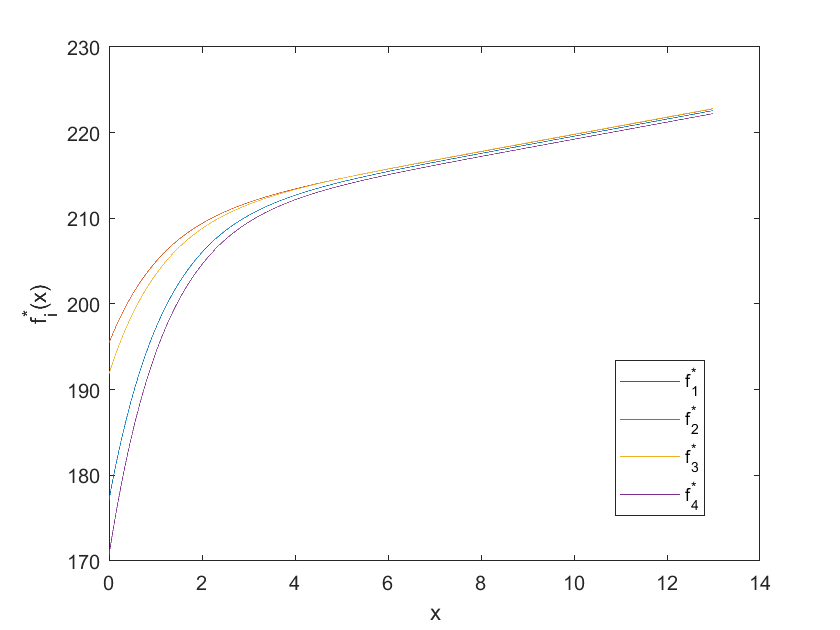}}
  \caption{$f_i^*(x)$ for Example \ref{fourorder2same}.}
  \label{fig:4order2same}
\end{figure}

\end{example}

\begin{example}\label{exam4allsame}
Let use try another example with the same intensity to the next claim.
$c=21$, $\delta=0.1$, $\beta=1$, $(\pi_1,\pi_2,\pi_3,\pi_4)=(0.5,0.2,0.2,0.1)$.
The intensity matrix $\mathbf{T}$ is
\begin{equation*}
\left(\begin{array}{ccccc}
    -\lambda_1 & \lambda_{12} & \lambda_{13} & \lambda_{14} &\lambda_{15} \\
    \lambda_{21} & -\lambda_2 & \lambda_{23} & \lambda_{24} &\lambda_{25} \\
    \lambda_{31} & \lambda_{32} & -\lambda_{3} & \lambda_{34}&\lambda_{35}\\
    \lambda_{41} & \lambda_{42} & \lambda_{43} & -\lambda_{4} &\lambda_{45} \\
\end{array}\right)=
\left(\begin{array}{ccccc}
    -16 & 7 & 3 & 1 &5 \\
    4 & -12 & 1 & 2 &5 \\
    0 & 1 & -7 & 1 &5 \\
    1 & 1 & 3 & -10 &5 \\
\end{array}\right).
\end{equation*}
The optimal phase-wise barriers are shown in Table \ref{tab:4order3}. In this example,  all phases have the same highest intensity to the next claim and  all phases have the same highest barrier. The figure in \ref{fig:4orderallequal} shows that the value of $f_1^*(x),f_2^*(x),f_3^*(x),f_4^*(x)$ equals to each other.
\begin{table}[h]
  \centering
\begin{tabular}{lcccc}
  \hline
  Phase $i$ & 1 & 2 & 3 & 4 \\\hline
  The intensity to the next claim & 5&5&5&5\\\hline
 % Expected  time of the next claim $\mathcal{T}_i$ & 0.440876 & 0.399397 & 0.483191 & 0.445392 \\
  Optimal phase-wise barrier $b_i^*$ & 10.611 & 10.611 & 10.611 & 10.611 \\
  \hline
\end{tabular}
  \caption{ $b_i^*$ with the same highest intensity}
  \label{tab:4order3}
\end{table}

\begin{figure}[htb]
\center{\includegraphics[width=0.7\textwidth]{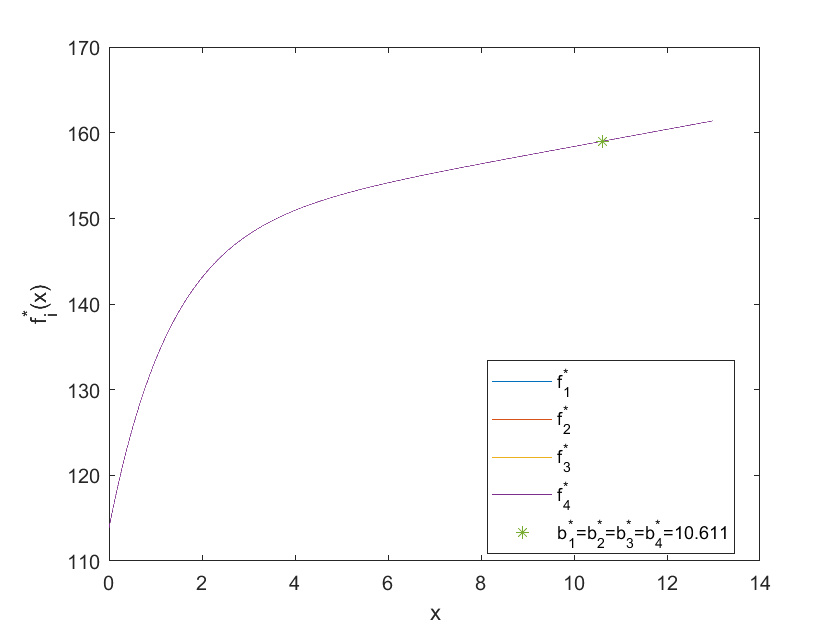}}
  \caption{$f_i^*(x)$ for Example \ref{exam4allsame}.}
  \label{fig:4orderallequal}
\end{figure}
\end{example}

\subsection{The phase with the highest barrier has the highest intensity to the claim}
\label{sec:sectionjust}

From the above numerical experiments, an interesting phenomenon can be brought out:
\begin{proposition}
The phase with the the highest barrier must be the phase with highest intensity to the next  claim.
\end{proposition}

Until now, we know that the optimal phase-wise barrier strategy is optimal among all admissible policies, not just the case of 2--order phase-type distribution. Thus,
the phase-wise optimal barrier strategy's corresponding value function solves the following HJB equation
\begin{equation*}
\max\left\{cf_i^*(x)+\sum_{j=1,j\neq i}^{n+1}\lambda_{ij}f_j^*(x)-(\lambda_i+\delta)f_i^*(x),1-{f_i^*}'(x)\right\}=0, i=1,\ldots, n,
\end{equation*}
where $f_{n+1}^*(x)=\sum_{i=1}^n\pi_i\mathbb{E}[f_i(x-Y)]$.
 Without loss of generality, we assume that $0<b_1<b_2<\ldots<b_n$.
At $x=b_n$,
\begin{equation}\label{justifiability1}
c{f_i^*}''(b_n)+\sum_{j=1, j\neq i}^{n+1}\lambda_{ij}{f_j^*}'(b_n)-(\lambda_i+\delta){f_i^*}'(b_n)\le 0, \quad i<n.
\end{equation}
\begin{equation}\label{justifiability2}
c{f_n^*}''(b_n)+\sum_{j=1,j\neq n}^{n+1}\lambda_{nj}{f_j^*}'(b_n)-(\lambda_n+\delta){f_n^*}'(b_n)=0.
\end{equation}
Notice that for all $i\le n,$  ${f_i^*}'(b_n)=1$ and ${f_i^*}''(b_n)=0$. Combining this with \eqref{justifiability1}, we get that
\begin{equation}\label{justifiability3}
{f_{n+1}^*}'(b_n)\le 1+\frac{\delta}{\lambda_{i,n+1}}, \quad i<n.
\end{equation}
On the other hand, by \eqref{justifiability2}, we get that
\begin{equation}\label{justifiability4}
{f_{n+1}^*}'(b_n)=1+\frac{\delta}{\lambda_{n,n+1}}.
\end{equation}
Combining \eqref{justifiability3} and \eqref{justifiability4}, we get that $\lambda_{i,n+1}\le \lambda_{n,n+1}$.
Eventually, we get that the $n^{\mbox{th}}$ phase has the highest intensity to the next claim, which verifies that the phase with the highest barrier has the highest intensity to the next claim.

\section{Conclusion}
In this paper, we study the optimal dividend problem when the interclaim times follow the $n$--order $(n\ge 2$) phase type distribution. By direct calculation, we theoretically prove that the optimal strategy is a  phase-wise barrier strategy in the case of 2--order distributed interclaim times. On the other hand, for the case of $n$--order distributed interclaim times,  since there are two many parameters in the matrix, the direct calculation is impossible. We bring up a numerical algorithm to show the optimality of the phase-wise barrier strategy.

About the optimal barrier's size comparison of different phases,   we also find that in the case of $n(n\ge 3)$--order, the phase with the highest is the phase with the highest intensity to the next claim, which coincides with the same conclusion in the case of $2$--order.

\section*{Acknowledgements}

Part of this work was done when the first author were visiting Professor Hansj\"org Albrecher. Thanks for his inspiration and guidance.

%% The Appendices part is started with the command \appendix;
%% appendix sections are then done as normal sections
%% \appendix

%% \section{}
%% \label{}

%% If you have bibdatabase file and want bibtex to generate the
%% bibitems, please use
%%
%  \bibliographystyle{elsarticle-harv}
 % \bibliography{}

\begin{thebibliography}{00}

%% \bibitem[Author(year)]{label}
%% Text of bibliographic item

  \bibitem[Albrecher et al.(2017)]{Albrecher2017}
    Albrecher, H., Azcue, P., \& Muler, N. (2017). Optimal dividend strategies for two collaborating insurance companies. {\it Advances in Applied Probability}, 49(2), 515--548.

  \bibitem[Albrecher et al.(2005)]{Albrecher2005}
    Albrecher, H., Claramunt, M.M., \& M\'armol, M. (2005). On the distribution of dividend payments in a Sparre Andersen model with generalized Erlang(n) interclaim times. {\it Insurance: Mathematics and Economics}, 37(2 SPEC. ISS.), 324--334.

  \bibitem[Albrecher \& Hartinger(2006)]{AlbrecherHartinger2006}
    Albrecher, H., \& Hartinger, J. (2006). On the non-optimality of horizontal barrier strategies in the Sparre Andersen model. {\it Hermis Journal of Computer Mathematics and Its Applications}, 7, 109--122.

  \bibitem[Asmussen \& Albrecher(2010)]{AsmussenAlbrecher2010}
    Asmussen, S., \& Albrecher, H. (2010). {\it Ruin Probabilities (Second Edition)}. New Jersey: World Scientific.

  \bibitem[Asmussen \& Taksar(1997)]{AsmussenTaksar1997}
    Asmussen, S., \& Taksar, M. (1997). Controlled diffusion models for optimal dividend pay-out. {\it Insurance: Mathematics and Economics}, 20(1), 1--15.

  \bibitem[Azcue \& Muler(2005)]{AzcueMuler2005Optimal}
    Azcue, P., Muler, N. (2005). Optimal reinsurance and dividend distribution policies in the Cram\'er-Lundberg model. {\it Mathematical Finance}, 15(2), 261--308.

  \bibitem[Belhaj(2010)]{Belhaj2010}
    Belhaj, M. (2010). Optimal dividend payments when cash reserves follow a jump-diffusion process. {\it Mathematical Finance}, 20(2), 313--325.

  \bibitem[Bladt(2005)]{Bladt2005}
    Bladt, M. (2005). A review on phase-type distributions and their use in risk theory. {\it Astin Bulletin}, 35(1), 145--161.

  \bibitem[De Finetti(1957)]{DeFinetti1957}
    De Finetti, B. (1957). Su un' impostazione alternativa dell teoria collettiva del rischio. In: {\it Transactions of the XVth International Congress of Actuaries}, New York, (II), 433--443.

  \bibitem[Gerber \& Shiu(2006)]{GerberShiu2006}
    Gerber, H.U., \& Shiu, E.S.W. (2006). On optimal dividend strategies in the compound Poisson model. {\it North American Actuarial Journal}, 10(2), 76--93.

  \bibitem[Liu et al.(2020)]{Liu2020Dividend}
    Liu, Y., Liu, Z., \& Liu, G. (2020). Optimal dividend problems for Sparre Andersen risk model with bounded dividend rates. {\it Scandinavian Actuarial Journal}, 2020(2), 128--151.

  \bibitem[Mishura \& Schmidli(2012)]{MishuraSchmidli2012}
    Mishura, Y., \& Schmidli, H. (2012). Dividend barrier strategies in a renewal risk model with generalized Erlang interarrival times. {\it North American Actuarial Journal}, 16(4), 493--512.

\end{thebibliography}

%% else use the following coding to input the bibitems directly in the
%% TeX file.

%%\begin{thebibliography}{00}

%% \bibitem[Author(year)]{label}
%% Text of bibliographic item

%%\bibitem[ ()]{}

%%\end{thebibliography}
\end{document}